\documentclass[11pt]{article}
\usepackage[utf8]{inputenc}
\usepackage{amsmath,amsfonts,amssymb}
\usepackage{amsthm}
\usepackage{latexsym}
\usepackage[noadjust]{cite}
\usepackage{graphicx}
\usepackage{hyperref}
\usepackage{xcolor}
\usepackage{dirtytalk}
\usepackage{mathtools}
\usepackage[margin=1in]{geometry}
\usepackage{thm-restate}
\usepackage{enumitem}
\usepackage[linesnumbered,ruled]{algorithm2e}

\newtheorem{theorem}{Theorem}[section]

\newtheorem{lemma}[theorem]{Lemma}
\newtheorem{observation}[theorem]{Observation}

\theoremstyle{definition}
\newtheorem{definition}[theorem]{Definition}
\newtheorem{remark}[theorem]{Remark}

\usepackage[nameinlink]{cleveref}
\Crefname{claim}{Claim}{Claims}
\crefname{theorem}{Theorem}{Theorems}
\crefname{claim}{Claim}{Claims}
\crefname{proposition}{Proposition}{Propositions}
\crefname{definition}{Definition}{Definition}
\crefname{lemma}{Lemma}{Lemma}
\crefname{corollary}{Corollary}{Corollaries}
\crefname{ineq}{inequality}{inequalities}
\Crefname{equation}{Equation}{Equations}

\newcommand{\F}{\mathbb{F}}
\newcommand{\N}{\mathbb{N}}
\newcommand{\E}{\mathbb{E}}
\newcommand{\R}{\mathbb{R}}
\newcommand{\calC}{\mathcal{C}}
\newcommand{\calD}{\mathcal{D}}
\newcommand{\calM}{\mathcal{M}}
\newcommand{\calP}{\mathcal{P}}
\newcommand{\calB}{\mathcal{B}}

\newcommand{\eps}{\varepsilon}

\newcommand{\ra}{\rightarrow}
\newcommand{\la}{\leftarrow}
\newcommand{\wt}{\text{wt}}

\newcommand{\map}{\varphi}
\newcommand{\Emp}{\mathrm{Emp}}
\newcommand{\DKLq}{D_{\text{KL}_q}}
\newcommand{\rank}{\textnormal{rank }}
\newcommand{\tr}{\textnormal{tr}}

\newif\ifdraft
\drafttrue

\title{Pseudorandom Linear Codes Are List Decodable to Capacity}
\author{Aaron (Louie) Putterman\thanks{Supported under the Simons Investigator Fellowship of Boaz Barak, NSF grant DMS-2134157, DARPA grant W911NF2010021, and DOE grant DE-SC0022199.} \\ Harvard University \\ \texttt{aputterman@g.harvard.edu} \and Edward Pyne\thanks{Supported by an Akamai Presidential Fellowship} \\ MIT\\ \texttt{epyne@mit.edu}}

\begin{document}
\begin{titlepage}
\maketitle
\begin{abstract}
    We introduce a novel family of expander-based error correcting codes. These codes can be sampled with randomness linear in the block-length, and achieve list-decoding capacity (among other local properties). Our expander-based codes can be made starting from any family of sufficiently low-bias codes, and as a consequence, we give the first construction of a family of algebraic codes that can be sampled with linear randomness and achieve list-decoding capacity. We achieve this by introducing the notion of a \textit{pseudorandom} puncturing of a code, where we select $n$ indices of a base code $C\subset \F_q^m$ via an expander random walk on a graph on $[m]$. Concretely, whereas a random linear code (i.e. a truly random puncturing of the Hadamard code) requires $O(n^2)$ random bits to sample, we sample a pseudorandom linear code with $O(n)$ random bits. We show that pseudorandom puncturings satisfy several desirable properties exhibited by truly random puncturings. In particular, we extend a result of (Guruswami Mosheiff FOCS 2022) and show that a pseudorandom puncturing of a small-bias code satisfies the same local properties as a random linear code with high probability. As a further application of our techniques, we also show that pseudorandom puncturings of Reed Solomon codes are list-recoverable beyond the Johnson bound, extending a result of (Lund Potukuchi RANDOM 2020). We do this by instead analyzing properties of codes with large distance, and show that pseudorandom puncturings still work well in this regime. 
\end{abstract}
\vfill

\thispagestyle{empty}
\end{titlepage}
\newpage
\section{Introduction}
Random linear codes (RLCs) are a fundamental tool in coding theory because of their many favorable combinatorial properties. RLCs attain near-optimal distance and list decodability with high probability, while still maintaining a linear structure (and can thus efficiently encode messages). However, the main drawback of RLCs is that there is no known algorithm for efficient decoding or list-decoding. This is often attributed to a lack of structure in the codes, resulting from the fact that the codewords are chosen uniformly at random. Indeed, viewing an RLC from the perspective of its generator matrix $G \in \F_2^{n \times Rn}$ (where $R$ is the design rate), each entry in this matrix is chosen uniformly at random, thus requiring $\Omega(n^2)$ truly random bits. 

While there existed codes that achieve some favorable properties of RLCs using much less randomness (for instance the Toeplitz codes which achieve a near-optimal distance tradeoff using $O(n)$ randomness), there did not exist $O(n)$-randomness constructions of codes achieving list-decoding capacity until the work of Guruswami and Moshieff~\cite{GM21}. Starting from the observation that an RLC is equivalent to a random puncturing of a Hadamard code, the work of \cite{GM21} showed that taking a random puncturing of any code of sufficiently low-bias or sufficiently large distance maintains a form of local similarity to an RLC. Local similarity in this context refers to \emph{local properties} of a code, which includes characterizations like list-decodability. Local properties are characterized by \emph{not} having certain sets of a small number of bad codewords. 

One consequence of this result is that \cite{GM21} were able to show that for every $n$, one can sample a code $\calC \subseteq \F_q^n$, each generated with $O(n)$ randomness that is list-decodable to capacity. This follows by using sufficiently low-bias codes of length $O(n)$, and correspondingly choosing a random subset of size $n$ of the indices. Choosing such a subset can be shown to require $O(n)$ randomness. However, this construction has 3 potential drawbacks:
\begin{enumerate}
    \item The mother codes of sufficiently low-bias are intricate constructions. 
    \item The subsets of indices that are chosen are not structured. 
    \item The random puncturing procedure only achieves $O(n)$ randomness when the mother code is of length $O(n)$.
\end{enumerate}

\subsection{Our Contributions}
We introduce the notion of a \emph{pseudorandom} puncturing, and show codes generated in this fashion exhibit several desirable properties exhibited by truly random puncturings. In a truly random puncturing, we choose $n$ indices of the code to preserve out of the $m$ indices of the mother code. These $n$ indices are chosen i.i.d. uniformly from $[m]$. One way to view such a puncturing is that we do a length $n$ walk on the complete graph over $m$ vertices, and read off the sequence of vertices from this walk as the indices of our puncturing. Performing such a random walk will then necessitate $n \cdot \log m$ random bits. 

In a pseudorandom puncturing, we replace the complete graph on $m$ vertices with a sufficiently-expanding $d$-regular expander on $m$ vertices. Now we perform a length $n$ random walk on this expander, which only requires $O(\log m + n \cdot \log d)$ random bits, and (as before) let the punctured code be defined by the indices of the walk. Even when taking the mother code to be a Hadamard code of length $2^{O(n)}$, we can sample a pseudorandom puncturing with $O(n)$ random bits by having $d$ a constant independent of $n$.

We informally define pseudorandom linear codes in the case $q=2$:
\begin{definition}
    Given $n\in \N$, the \textbf{pseudorandom (binary) linear code} of block length $n$ and rate $R$ is generated by a (sufficiently pseudorandom) length $n$ puncturing of the Hadamard code $\calD\subset \F^{2^{Rn}\times Rn}$.
\end{definition}

We give an immediate corollary of our main theorem:
\begin{theorem}[Optimal list-decoding with linear randomness]\label{thm:optLD}
    For every $\rho\in (0,1/2)$ and $\eps>0$, pseudorandom linear codes are $(\rho,L)$-list decodeable up to rate $R\leq 1-h_2(\rho)-\eps$ for $L:=O(1/\eps)$. Moreover, these codes can be sampled with $O \left (n/\eps \right )$ random bits.
\end{theorem}

Moreover, our results directly extend~\cite{GM21}, in that we show pseudorandom puncturings of all sufficiently small bias codes also achieve list-decoding capacity. Recall that for a binary code $\calC\subset \F_2^n$, the bias is defined as 
\[
\max_{c\in \calC}\frac{|2|c|-n|}{n}.
\]

As a consequence, we can now construct pseudorandom linear codes (with generator matrix $G$) list-decodable to capacity that are sampled with $O(n)$ random bits such that the rows of $G$ come from a low-bias mother code, and the columns of $G$ are a pseudorandom subset. 

In particular, as noted by \cite{GM21}, we can take the mother code $\calD$ to be a dual-BCH code, where every codeword encodes a low-degree polynomial over $\F_{2^{\ell}}$ by the trace of its evaluations over $\F_{2^{\ell}}$. In the setting of \cite{GM21}, a random puncturing of $\calD$ corresponds to codewords being evaluations over a \emph{random} subset of $\F_{2^{\ell}}$. However, by instead taking a \emph{pseudorandom} puncturing, the codewords are now evaluations over a more constrained subset of $\F_{2^{\ell}}$. In fact, we decrease the randomness required in this construction from $ \Omega(n \log n)$ to $O(n)$, while still preserving its list-decodability. 
We view it as an interesting open question to find expanders with sufficient algebraic structure such that they may make decoding in this scheme tractable.

More generally, we prove the following theorem: 

\begin{theorem}[More general case, informally]{\label{thm:informalGeneral}}
    Let $\calD \subseteq \F_q^m$ be a linear, sufficiently low-bias code. Let $\calC$ be a sufficiently pseudorandom puncturing of $\calD$. Then, $\calC$ is likely to have every \emph{monotone-decreasing, local property } that is typically satisfied by an RLC of similar rate. In fact, for every sufficiently low-bias mother code, we can pseudorandomly subsample the indices of our puncturing with randomness $O(nb)$, where $b$ is the locality of the property, even when the length of the mother code is exponential in $n$. 
\end{theorem}

We further illustrate the flexibility of the pseudorandom puncturing approach in two regimes. First, we observe (Theorem~\ref{thm:MLDUDecode}) that pseudorandom linear codes achieve capacity against the memoryless additive channel, extending the analogous result of~\cite{GM21}. 

Finally, we apply our techniques in a different regime: we partially derandomize the result of Lund and Potukuchi~\cite{LP20}, who show that random puncturings of Reed-Solomon codes can be list-recovered beyond the Johnson bound. 
\begin{definition}[Zero-Error List Recoverability]
    Let $\calC\in \F_q^n$ be a code. We say $\calC$ is $(\ell,L)$ \textbf{zero-error list recoverable} if for every collection of sets $A_1,\ldots,A_n$ with $|A_i|\leq \ell$ for all $i$, we have $|\{c\in \calC:c\in A_1\times \ldots \times A_n\}|\leq L$.
\end{definition}
We show that \emph{pseudorandom} puncturings of Reed Solomon codes are zero-error list recoverable beyond the Johnson bound:
\begin{theorem}[Zero-Error List Recovery of Reed Solomon Codes]\label{thm:RSintro}
    Given a prime power $q$ and $\eps \geq 1/\sqrt{q}$, there are Reed-Solomon codes of length $n$ and rate $\Omega(\eps/\log q)$ that can be sampled with $O(n)$ randomness that are $(\eps^{-2},O(\eps^{-2}))$-zero error list recoverable.
\end{theorem}

We do this by analyzing how pseudorandom puncturings work for codes with near-maximal distance. We note that \cite{GM21} analyzed random puncturings in the case of large distance as well, though their more structured analysis does not carry over to the regime of pseudorandom puncturings. 

\subsection{Overview}
In \Cref{sec:prelims} we recall concentration results for expander random walks, local properties of codes, and notation related to distributions on rows of a matrix. In \Cref{sec:basic} we prove a weaker version of \Cref{thm:optLD} with exponential list sizes to introduce our proof strategy. In \Cref{sec:main} we prove \Cref{thm:optLD}. In \Cref{sec:random} we prove random puncturings achieve capacity in the memoryless channel. In \Cref{sec:RS} we prove \Cref{thm:RSintro}, and in \Cref{app:expcons} we conclude a small derandomization of a result constructing unbalanced expanders. 

\section{Preliminaries}\label{sec:prelims}
We first introduce concepts required for the proofs.

\subsection{Properties of Expander Walks}
We recall some useful statements of properties of expander random walks. We reference the excellent survey of Hoory, Linial, and Wigderson~\cite{HL06}.

First, we reintroduce the definition of an expander, and that ones exist with good properties. Our results are not sensitive to the precise degree-expansion tradeoff, except in optimizing the constant factor on the number of bits required to sample.
\begin{definition}[Expander graphs \cite{HL06}]
    We say a graph $(G, V)$ is an $(m, d, \lambda)$-expander if $G$ is $d$-regular on $m$ vertices, and satisfies $|\lambda_2(G)|, |\lambda_m(G)| \leq \lambda d$. The notation $\lambda_i(G)$ refers to the $i$th eigenvalue of the adjacency matrix of $G$. 
\end{definition}

\begin{theorem}[Existence of near-optimal expanders \cite{Vad12}]\label{degreeExpansion}
For a fixed $d \in \N, \lambda$, there exist explicit constructions of $(m, d, \lambda)$-expanders for all $m$ large enough if $\lambda \leq \frac{4}{\sqrt{d}}$.
\end{theorem}
\begin{remark}
    To take a $n$-step random walk on an $(m, d, \lambda)$-expander takes $O(\log m) + n \cdot O(\log d)$ random bits. Using known degree-expansion trade-offs and the existence of good explicit expanders, $O(\log m)+cn\cdot \log(1/\lambda) + O(n)$ random bits suffices for some constant $c\geq 2$.
\end{remark}

We will also make use of the non-equal expander hitting set lemma, which states that a random walk on an expander lies inside a sequence of sets with probability approximately the product of the sets densities. Our analysis relies on the ability of the sets to differ at each timestep.

\begin{theorem}[Non-equal expander hitting-set lemma, \cite{HL06}, Theorem 3.11]\label{thm:EHS}
    Let $B_1, B_2, \dots B_t$ be vertex sets of densities $\beta_1, \dots \beta_t$ in an $(m, d, \lambda)$-graph $G$. Let $X_1, \dots X_n$ be an $n$-step random walk on $G$. Then, 
    \[
    \Pr[\forall {i\in [n]}, X_i \in B_i] \leq \prod_{i = 1}^{n-1} \left ( \sqrt{\beta_i \beta_{i + 1}} + \lambda \right) \leq \left(\max_{i}\beta_i+\lambda\right)^{n-1}.
    \]
\end{theorem}

Additionally, we will require the expander Chernoff bound \cite{G93}.

\begin{theorem}[Expander Chernoff bound]\cite{G93}\label{thm:expanderChernoff}
    Let $G$ be an $(m,d,\lambda)$ regular graph. Let $B\subset [m]$ be a set with density $\mu:=|B|/m$. Let $X_1,\ldots,X_n$ be an $n$-step stationary random walk on $G$. Then,
    \[
    \Pr \left [  \left |\sum_{i = 1}^n \mathbb{I}[X_i\in B] - n\mu \right | > n\eps \right ] \leq 2 e^{-\Omega((1 - \lambda)n \eps^2)}.
    \]
\end{theorem}

\subsection{Pseudorandom Puncturing}

Of primary importance in this paper will be the notion of a pseudorandom puncturing:
\begin{definition}[$\lambda$-pseudorandom puncturing]\label{def:pseudorandomPuncturing}
Given a prime power $q$ and $m,n\in \N$, a \textbf{$(m\ra n)$ $\lambda$-pseudorandom puncturing map} $\map:\F_q^m\ra\F_q^n$ is a random function obtained by taking an expander $G=([m],E)$ satisfying $\lambda(G)\leq \lambda$ and taking a length $n$ random walk. Letting the vertex labels of the walk be $(i_1,\ldots,i_n)$, we define the map by
\[\map(u = (u_1,\ldots,u_m)) = (u_{i_1},\ldots,u_{i_n}).\]
For $j\in [n]$ let $\map_j=i_j$ be the $j$th index of the map. Given a code $\calD \subset \F_q^m$, we say $\calC$ is a \textbf{$\lambda$-pseudorandom puncturing} of $D$ if
\[\calC:=\map(\calD) = \{\map(u):u\in \calD\}.
\]
The \textbf{design rate} of $\calC$ is $R=\log_q|\calD|/n$.
\end{definition}
We note that our pseudorandom puncturing map places no constraints on the expander beyond its spectral gap.

\cite{GM21} show that the rate of a random puncturing (of a small-bias code) is equal to the design rate with high probability. We extend this result to pseudorandom puncturing, subject to mild constraints on the parameter $\lambda$.
\begin{lemma}[Actual rate equals design rate with high probability]\label{lem:designRate}
    Let $\calD \subseteq \F_m^q$ be a linear code of $\eta$-optimal distance, and let $\calC$ be a length-$n$ $\lambda$-pseudorandom puncturing of $\calD$, of design rate $R \leq 1 - \log_q(1 + \eta q+\lambda q) - \eps$. Then, with probability at least $1 -q^{-\eps n}$, the rate of $\calC$ is equal to its design rate.
\end{lemma}
\begin{proof}
    The event that the rate is less than the design rate occurs if there is some nonzero codeword $u\in \calD$ such that $\map(u)=0$. Fixing $u\in \calD$, let $T\subset [m]$ be the coordinates on which $u$ is zero. We have
    \[\frac{|T|}{m} = 1-\wt(u) \leq \frac{1}{q}+\frac{q-1}{q}\eta \leq \frac{1}{q}+\eta.
    \]
    Then 
    \[\Pr[\map(u)=0] = \Pr[\map_1 \in T \wedge \ldots \wedge \map_n \in T] \leq \left(\frac{1}{q}+\eta+\lambda\right)^n = q^{-n(1-\log_q(1+q\eta +q\lambda))}
    \]
    where the first inequality comes from \Cref{thm:EHS}. Then a union bound over the $q^{Rn}$ codewords completes the proof.
\end{proof}

\subsection{Properties of Codes}
As in \cite{GM21} and \cite{MR20}, we will be proving a result that generalizes to a wide class of \emph{properties} of codes. 

Our results will rely on the distance and bias of codes.

\begin{definition}[Bias and distance]
Let $\calD \subseteq \F_q^m$ be a linear code.
\begin{enumerate}
    \item We say that $\calD$ has \textbf{$\eta$-optimal distance} if the weight of every codeword is bounded below by $(1 - 1/q)(1 - \eta)$. That is,
    \[
    \forall c \in \calD, \wt(c) \geq (1 - 1/q)(1 - \eta).
    \]
    \item We say that $\calD$ is \textbf{$\eta$-biased} if for every non-zero codeword $c \in \calD$, for every $a \in \F_q^{*}$:
    \[
    \left | \sum_{i = 1}^m \omega^{\tr(a \cdot c_i)}  \right | \leq m \eta.
    \]
    Here, we use that $\omega = e^{2\pi i / p}$, (where $q$ is a power of a prime $p$), and $\tr: \F_q \ra \F_p$ is defined as 
    \[
    \tr(x) = \sum_{i = 0}^{r-1} x^{p^i},
    \]
    where $r = \log_p q$. 
\end{enumerate}
\end{definition}
\begin{remark}
    A code that is $\eta$-biased has $\eta$-optimal distance, and most of our analysis uses only this property (though we use results of~\cite{GM21} which rely on the bias condition).
\end{remark}

Now, we will first introduce a few specific examples of properties, and then the more general definition for which our result will ultimately apply. 

We first define $\rho$-clustered. We note that $\wt(x)$ is the normalized Hamming weight of $x$.
\begin{definition}[$\rho$-clustered \cite{GM21}]\label{def:rhoClustered}
Fix $\rho \in [0, 1]$. We say that a set of vectors $W \subseteq \F_q^n$ is $\rho$-\textbf{clustered} if there exists a $z \in \F_q^n$ such that $\wt(w - z) \leq \rho$ (equivalently, $w\in H(z,\rho n)$) for all $w \in W$.
\end{definition}
We recall the observation of \cite{GM21} that this definition gives a clean characterization of list decodability:
\begin{observation}
    A code $C \subseteq \F_q^n$ is $(\rho, L)$-list decodable if and only if it does not contain a $\rho$-clustered set of codewords of size $L + 1$.
\end{observation}

Both list-decodability and list-recoverability are special cases of properties of codes \cite{GM21,MR20}.
\begin{definition}[Properties of a code]\label{def:property}
A property $\calP$ of length $n$ linear codes over $\F_q$ is a collection of linear codes in $\F_q^n$. For such a code $\calC$, if $\calC \in \calP$, then we say that $\calC$ satisfies property $\calP$. A property $\calP$ is said to be monotone-increasing if $\calP$ is upwards closed with respect to containment. 
\end{definition}

\begin{definition}[Local and row-symmetric properties]
Let $\calP$ be a monotone-increasing property of linear codes in $\F_q^n$.
\begin{enumerate}
    \item If, for a fixed $b \in \N$, there exists a family $\calB_{\calP}$ of sets of words, such that every $B \in \calB_{\calP}$ is a subset of $\F_q^n$, $|B| \leq b$, and 
    \[
    \calC \text{ satisfies } \calP \iff \exists B \in \calB_{\calP}: B \subseteq \calC,
    \]
    then we say $\calP$ is a $b$-local property. 
    \item If, whenever a code $\calC$ satisfies $\calP$ and $\pi$ is a permutation on $\{1, \dots n \}$, the code $\{\pi x | x \in \calC \}$ also satisfies $\calP$, then we say that $\calP$ is row-symmetric. $\pi x$ in this notation refers to permuting the entries of a vector of length $n$ according to the permutation $\pi$.
\end{enumerate}
    
\end{definition}

Note that the property of being \emph{not} $(\rho, L)$ list-decodable is a $L$-local row-symmetric property. We will use this in our result. 

\begin{definition}[Threshold of a property]\label{def:threshold}
    For $\calP$ over $\F_q^n$, we will let
    \[
    \mathrm{RLC}(\calP) = \min \left \{ R \in [0, 1] | \Pr\left [ \text{RLC of length } n, \text{ rate } R, \text{ domain } \F_q \text{ satisfies } \calP \right ] \geq 1/2 \right \}.
    \]
\end{definition}

This definition is motivated by the following observation which was proved in \cite{MR20}.
\begin{theorem}[Sharp threshold behavior \cite{MR20}]\label{thm:RLCthresholds}
    Let $\calC \subseteq \F_q^n$ be a random linear code of rate $R$ and let $\calP $ be a monotone-increasing, $b$-local, and row-symmetric property over $\F_q^n$, where $\frac{n}{\log_q n} \geq \omega_{n \ra \infty} \left ( q^{2b} \right )$. Then, for every $\eps > 0$, the following hold:
    \begin{enumerate}
        \item If $R \leq \mathrm{RLC}(\calP) - \eps$
        \[
        \Pr[\calC \text{ satisfies } \calP] \leq q^{-n(\eps - o_{n \ra \infty}(1))}.
        \]
        \item If $R \geq \mathrm{RLC}(\calP) + \eps$
        \[
        \Pr[\calC \text{ satisfies } \calP] \geq 1 - q^{-n(\eps - o_{n \ra \infty}(1))}.
        \]
    \end{enumerate}
\end{theorem}

Because of \Cref{thm:RLCthresholds}, it suffices to show that the local behavior of a pseudorandom puncturing is similar to that of a random linear code. From there, we can invoke this result about thresholds to conclude whether or not a property $\calP$ is satisfied with high probability.

We will also be using the definition of $q$-ary entropy in this paper.
\begin{definition}[$q$-ary entropy]
    For $x \in [0,1]$ the $q$-ary entropy is defined to be 
    \[
    h_q(x) = -x \log_q(x) - (1-x)\log_q(1-x) + x\log_q(q-1).
    \]
\end{definition}

We then recall the statement of the Vazirani XOR Lemma used in prior work:
\begin{lemma}[Vazirani's XOR lemma \cite{Gol11,GM21}]\label{lem:xor}
    Let $\sigma$ be a distribution over $\F_2^b$ such that for every $y \in \F_q^b / \{ 0 \}$ we have that $\frac{1 - \eta}{2} \leq \Pr_{x \sim \sigma }[\langle x, y \rangle = 1] \leq \frac{1 + \eta}{2} $. Then, $\sigma$ is $(2^b \cdot \eta)$-close in total-variation distance to the uniform distribution over $\F_2^b$.
\end{lemma}

\subsection{Empirical Distributions}

In order to eventually prove a tight bound on list sizes, we will need the notion of empirical distributions (types) from \cite{CT01} \cite{GM21}. 
\begin{definition}[Empirical Distribution]
For a vector $a \in \F_q^n$, the \emph{empirical distribution} $\Emp_a$ assigns probability $\forall x \in \F_q$:
\[
\Emp_a(x) = \frac{\text{number of instances of } x \text{ in } a}{n}.
\]
This extends to a matrix $A \in \F_q^{n \times b}$ by defining $\forall x \in \F_q^b$
\[
\Emp_A(x) = \frac{\text{number of instances of } x \text{ in rows of } A}{n}.
\]

Note in this second case, $\Emp_A$ is a distribution over vectors $ \in \F_q^b$. 
\end{definition}

For convenience, we will also introduce the set of matrices for a distribution. 

\begin{definition}[Matrices of a distribution]
    Let $\tau$ be a distribution over $\F_q^b$. For $n \in \N$, 
    \[
    \calM_{n, \tau} = \left \{ A \in \F_q^{n \times b} | \Emp_A = \tau \right \}.
    \]
\end{definition}

Lastly, we will consider all sequences of samples that lead to a specific empirical distribution. 

\begin{definition}[Type class]\label{def:typeClass}
    The \emph{type class} of a distribution $\tau$ over $\F_q^b$ (denoted $T(\tau)$) is the set of all sequences $\left ( x_i \right )_{i = 1}^n$ in $\left ( \F_q^b \right)^n$ such that for the matrix 
    \[
    X = \begin{bmatrix}
        -&x_1&- \\
        \vdots & \vdots & \vdots \\
        -&x_n&-
    \end{bmatrix},
    \]
    we have that $\Emp_X = \tau$.
\end{definition}

When measuring the distance between distributions, we will consider the $q$-ary KL-divergence.
\begin{definition}[$q$-ary KL divergence]
    The $q$-ary KL divergence of two distributions $\tau, \sigma$ over a set $S$ is defined as 
    \[
    \DKLq(\tau \Vert \sigma) = \sum_{s \in S} \tau(s) \log_q \frac{\tau(s)}{\sigma(s)}.
    \]
\end{definition}

\section{List Decodability of Pseudorandom Linear Codes}\label{sec:basic}
In this section we give an outline of our main proof technique. For simplicity, we do not attain optimal list-size, and consider only the regime of list-decoding (as opposed to more general local properties). Our proof closely follows that of Theorem 7 of~\cite{GM21}. We state our initial result:
\begin{theorem}
    Let $\rho\in (0,1/2)$ and $L\in \N$. Then there exist $\eta(L)>0$ and $\lambda(L)>0$ and $\eps(L)>0$ with $\eps(L)$ tending to $0$ as $L\ra\infty$ such that the following holds. Let $\calD \subset \F^m_2$ be an arbitrary linear $\eta$-biased code, and let $\calC\subset \F_2^{n}$ be a $\lambda$-pseudorandom $n$-puncturing of $\calD$ of design rate $R \leq 1 - h_2(\rho) - \eps$. Then $\calC$ is $(\rho,L)$-list-decodable with high probability as $n\ra \infty$ and requires $O(n/\eps)$ random bits to construct.
\end{theorem}
\newcommand{\fail}{\calC \text{ fails to be }(\rho,L)\text{-list decodable}}
\newcommand{\rc}{$\rho$-clustered}
\newcommand{\st}{\text{ s.t. }}
\begin{proof}
     Let $\map$ be an $(m\ra n)$ $\lambda$-pseudorandom puncturing, for $\lambda$ to be chosen later, and let $\calC:=\map(\calD)$. Let $b:=\lceil \log L+1\rceil$. Recall that $\calC$ fails to be list decodeable if there exist $L+1$ codewords that are $\rho$-clustered (Definition \ref{def:rhoClustered}). Recalling an argument first used in~\cite{ZP81}, a necessary condition for this is for $\calC$ to contain $b$ linearly independent (L.I.) $\rho$-clustered codewords. Thus:
    \begin{align*}
        \Pr[\fail] &\leq \Pr[\exists v_1,\ldots,v_b \in \calC \text{ that are L.I. and \rc}]\\
        &\leq \sum_{\substack{u_1,\ldots,u_b \in \calD\\ \text{ lin. indep.}}}\Pr[\map(u_1),\ldots,\map(u_b) \text{ are \rc}].
    \end{align*}
    Here, we have used the substitution that $v_i = \map(u_i)$, where $u_i \in \calD$. Because $v_1, \dots v_b$ are linearly independent, this means that $u_1, \dots u_b$ must also be linearly independent.  Note that this sum is over $\leq |\calD|^b \leq 2^{bRn}$ terms. Now fix arbitrary, linearly independent $u_1,\ldots,u_b \in \calD$ and let 
    \[B = \begin{bmatrix}| & & |\\u_1 & \ldots & u_b\\ | & & |\end{bmatrix} \in \F_2^{m\times b}, \quad\quad A = \begin{bmatrix}| & & |\\\map(u_1) & \ldots & \map(u_b)\\ | & & |\end{bmatrix} \in \F_2^{n\times b}
    \]
    where $A$ is a random matrix defined in terms of the puncturing $\map$. We now note
    \begin{align*}
        \Pr[\map(u_1),\ldots,\map(u_b) \text{ are \rc}] &= \Pr[\exists z,y_1,\ldots,y_b \in H(z,\rho n) \st \forall i, y_i=\map(u_i)]\\
        &\leq \sum_{z\in \F_2^n}\sum_{y_1,\ldots,y_b\in H(z,\rho n)}\Pr[\forall i,\map(u_i)=y_i]
    \end{align*}
    Now fix arbitrary $z$ and $y_1,\ldots,y_b \in H(z,\rho n)$ (of which there are at most $2^{n+bh_2(\rho)n}$). Define the matrix
    \[Y = \begin{bmatrix}| & & |\\y_1 & \ldots & y_b\\ | & & |\end{bmatrix}.
    \]
    Finally, for $\sigma\in \F_2^b$ let $T_\sigma \subseteq [m]$ be defined as 
    \[T_\sigma := \{i \in [m]: B_i = \sigma\},\]
    where $B_i$ is the $i$th row of the matrix $B$. In words, each set $T_{\sigma}$ is the set of indices $i$ such that the $i$th row of $B$ equals $\sigma$.
    
    We first note that by \Cref{lem:xor}, for every $\sigma \in \F_2^b$ we have
    \[\tau_\sigma:=\frac{|T_\sigma|}{m}\leq 2^{-b}+2^b\eta \leq 2^{-b+1}\]
    where we use that $\eta \leq 2^{-2b}$ (and in the case of the Hadamard code we have $\eta = 0$). Then we have 
    \begin{align*}
        \Pr[\forall i,\map(u_i)=y_i] &= \Pr[\map_1 \in T_{Y_1},\ldots,\map_n \in T_{Y_n}]\\
        &\leq \left(\max_{\sigma} \tau_\sigma+\lambda\right)^{n-1} && \text{(\Cref{thm:EHS})}\\
        &\leq \left(2^{-b+2}\right)^n \cdot 2^{b-2},
    \end{align*}
    where the first equality follows from the definition of $\map$, and the final line follows from $\lambda \leq 2^{-b}$ and $\tau_i\leq 2^{-b+1}$ and changing the product to be over $n$ terms. Thus the entire expression is bounded as
    \begin{align*}
        \Pr[\fail] &\leq  2^{bRn}\cdot 2^n\cdot 2^{bh_2(\rho)n}\cdot (2^{-b+2})^n\cdot 2^{b-2}\\ 
        &\leq 2^{b(1-h_2(\rho)-4/b)n+3n+bh_2(\rho)n-bn+b-2} \\
        &= 2^{-n + b - 2} \ra 0,
    \end{align*}
    where we ultimately set $\eps = 4/b$.
    
    Note that this construction only requires $O(\log m + n \log d)$ random bits, where $d$ is the degree of the expander graph. $d$ must be chosen to satisfy $\lambda \leq 2^{-b}$, and by \Cref{degreeExpansion}, this can be done such that $O(\log m + n \log d) = O(nb)$.
\end{proof}

In the simplest case, we can take the mother code to the Hadamard code mapping messages of length $Rn$ to codewords of length $2^{Rn}$. The generator matrix for this Hadamard code is $\in \F_2^{2^{Rn} \times Rn}$. Choosing the starting vertex for the expander random walk in this case takes $Rn$ bits of randomness, and for every subsequent step, the amount of randomness required depends only on the degree of the expander. Ultimately, the pseudorandom puncturing results in a generator matrix of size ${n \times Rn}$. For a desired rate $1 - H(\rho) - \eps$, we take $b = 4 / \eps$. Correspondingly, we need $\lambda \leq 2^{-4 / \eps}$, which forces $\log d = \Omega(1 / \eps)$. Thus, we pay for the randomness linearly in $1 / \eps$.

However, because we set $\eps = \frac{4}{b}$, we get that $b = \frac{4}{\eps}$, meaning that the list size $L = 2^{\Omega(1 / \eps)}$, which is far from optimal. In the next section, we give a more careful argument that achieves optimal list sizes.

\section{Pseudorandom Puncturings Preserve Local Properties}\label{sec:main}
In this section, we give an analogue of the more detailed analysis presented in \cite{GM21} for the case of pseudorandom puncturings. 

In particular, we will show the following, and use it to conclude \Cref{thm:optLD}:
\begin{theorem}\label{thm:GM1}
    Let $q$ be a prime power, and let $\calP$ be a monotone-increasing, row-symmetric and $b$-local property over $\F_q^n$, where $\frac{n}{\log n} \geq \omega_{n \ra \infty}(q^{2b})$. Let $\calD \subseteq \F_q^m$ be a linear code. Let $\calC$ be a $\lambda = \frac{\eps \ln q}{8q^b}$-pseudorandom puncturing of $\calD$ of design rate $R \leq \text{RLC}(\calP) - \eps$ for some $\eps > 0$. Suppose that $\calD$ is $\eta = \frac{\eps b \ln q}{4q^{2b+1}}$-biased. Then,
    \[
    \Pr[\calC \text{ satisfies } \calP] \leq q^{(-\eps + o_{n \ra \infty}(1))n}.
    \]
\end{theorem}

At a high level, our proof has the following form:
\begin{enumerate}
    \item First, fix an $\eta$-biased code $\calD$, a distribution $\tau$ over $\F_q^b$, and a set of $b$ linearly independent columns in $\calD$, which we denote $(\calD)_{\text{res}}$. We show that if we sample rows of $(\calD)_{\text{res}}$ via a pseudorandom puncturing, we can upper bound the probability of our sampled rows having the same marginal probabilities as $\tau$. This bound will be in terms of $q$, the KL divergence between $\tau$ and the empirical distribution produced of $(\calD)_{\text{res}}$, and some error terms.
    That is, we will show (for specific conditions): 
    \[
    \Pr[\Emp_X = \tau] \leq q^{n \left ( -\DKLq(\tau \Vert \sigma) + \log_q \left ( 1 + \frac{\lambda q^b}{(1 - q^{2b} \eta)} \right ) + o_n(1) \right )},
    \]
    where $\Emp_X$ is the empirical distribution (under the pseudorandom puncturing) of the rows sampled from $(\calD)_{\text{res}}$, and $\sigma$ is the empirical distribution of the rows of $(\calD)_{\text{res}}$.

    \item Next, we invoke results from \cite{GM21} which characterize codes satisfying local properties in terms of the number of submatrices contained in the code that have a specific row distribution. We will use $\calM_{n, \tau}\subset \F_q^{n\times b}$ to denote the set of matrices with row distribution $\tau$. This result is independent of the puncturing procedure, and shows that it suffices to prove
    \[
    \E_{\calC} \left [ \left | \{X \subseteq \calC|X \in \calM_{n, \tau}\} \right | \right ] \leq q^{(H_q(\tau) - a(1-R) + a \eps)n}.
    \]

    \item Finally, we use item (1) to prove the bound from item (2) and conclude our proof.  That is, we will use the fact that a matrix $X\subseteq \calC$ is in $\calM_{n, \tau}$ only if $\Emp_X = \tau$. As we have strong bounds on this event from item (1), we can invoke a union bound and prove the desired result. 
\end{enumerate}

\subsection{Analysis}

First, we prove the following lemma (which is a pseudorandom version of a statement from \cite{CT01}):

\begin{lemma}\label{thm:biasedemp}
Let $D \in \F_q^{m \times b}$ be $b$ linearly independent codewords from an $\eta$-biased code of length $m$, where $\eta < q^{-2b}/4$. Further, let $\sigma$ be $\Emp_D$. Suppose that we sample rows of $D$ in accordance with a length $n$ $\lambda$-expander random walk over vertex set $[m]$, and place these as the rows in a matrix $X \in \F_q^{n \times b}$. Then for every distribution $\tau$ over $\F_q^b$,
\[
\Pr[\Emp_X = \tau] \leq q^{n \left ( -\DKLq(\tau \Vert \sigma) + \log_q \left ( 1 + \frac{\lambda q^b}{(1 - q^{2b} \eta)} \right ) + o_n(1) \right )}.
\]
\end{lemma}

\begin{proof}
We let $T(\tau)$ denote the type class of $\tau$ (see \Cref{def:typeClass}). In this context, we will let $P \in T(\tau)$ denote a specific sequence of samples in $\left ( \F_q^b \right )^n$, such that the marginals are $\tau$. We will let $P_i$ be an element in $\F_q^b$ corresponding to the $i$th sample of this sequence. From the perspective of the expander random walk, we will let $B_i$ denote the set of all vertices of the expander (i.e. indices from $[m]$) such that the corresponding sample (corresponding row of $D$) is $P_i$. We will let $\beta_i$ denote the density of $B_i$. We will let $X_1, \dots X_n$ denote the random walk over $[m]$, the rows of the mother code.

Then we have:
\begin{align*}
    \Pr[\Emp_X = \tau] &= \sum_{P \in T(\tau)} \Pr[\land_{i = 1}^n X_i \in B_i]\\
    &= \sum_{P \in T(\tau)} \Pr[\land_{i = 1}^n X_i \in B_i | X_{i - 1} \in B_{i -1 }]\\
    &\leq \sum_{P \in T(\tau)} \prod_{i = 1}^{n-1} \left ( \sqrt{\beta_i \beta_{i+1}} + \lambda \right )\\
    &\leq \sum_{P \in T(\tau)} \prod_{i = 1}^{n-1} \left (\beta_i \cdot q^{\log_q \left ( 1 + \frac{\lambda q^b}{(1 - q^{2b} \eta)} \right )}\right )\\
    &= \sum_{P \in T(\tau)} \prod_{i = 1}^{n-1} q^{\log_q \left ( 1 + \frac{\lambda q^b}{(1 - q^{2b} \eta)} \right )} \sigma(P_i)
\end{align*}
where the last inequality comes from the fact that (letting $\beta = \min_i\beta_i$):
\[\prod_i (\sqrt{\beta_i\beta_{i+1}}+\lambda) \leq \prod_i\beta_i(1+\beta^{-1}\lambda)\leq \left(\prod_i \beta_i\cdot q^{\log_q(1+\beta^{-1}\lambda)} \right).
\]
Then, because the mother code is $\eta$-biased and we have a selection of linearly independent codewords, we get that $\beta \geq q^{-b} \cdot \left ( 1 - q^{2b} \cdot \eta) \right )$ by Vazirani's XOR Lemma \cite{Gol11}. So, $\beta^{-1} \lambda \leq \frac{\lambda q^b}{(1 - q^{2b} \eta)}$.
Now we can bound $\Pr[\Emp_X = \tau]$. We see that 
\begin{align*}
    \Pr[\Emp_X = \tau] &\leq \sum_{P \in T(\tau)} \prod_{i = 1}^{n-1} q^{\log_q \left ( 1 + \frac{\lambda q^b}{(1 - q^{2b} \eta)} \right )} \sigma(P_i)\\
    &\leq \frac{q^b}{1- q^{2b}\eta} \sum_{P \in T(\tau)}\prod_{i = 1}^n q^{\log_q \left ( 1 + \frac{\lambda q^b}{(1 - q^{2b} \eta)} \right )} \sigma(P_i)\\
    &= \frac{q^b}{1- q^{2b}\eta} q^{n \log_q \left ( 1 + \frac{\lambda q^b}{(1 - q^{2b} \eta)} \right )} \cdot \sum_{P \in T(\tau)}\prod_{i = 1}^n \sigma(P_i)\\
    &= \frac{q^b}{1- q^{2b}\eta} q^{n \log_q \left ( 1 + \frac{\lambda q^b}{(1 - q^{2b} \eta)} \right )} \cdot q^{-\DKLq (\tau \Vert \sigma)n}
\end{align*}
where the last equality holds from the fact that $\sum_{P \in T(\tau)}\prod_{i = 1}^n \sigma(P_i) = q^{-\DKLq (\tau \Vert \sigma)n}$ (\cite{CT01}, Theorem 11.1.4). The second inequality comes from upper bounding $\frac{1}{\sigma(P_i)}$, so we can extend the product to $n$ terms. By our choice of $\eta$, the leading term $\frac{q^b}{1- q^{2b}\eta}$ is $O(q^b)$, and is thus $q^{n \cdot o_n(1)}$.
\end{proof}

Using \Cref{thm:biasedemp}, we can now prove the following key theorem:
\begin{lemma}\label{lemma:GM510}
    Fix a distribution $\tau $ over $\F_q^b$. Let $B \in F_q^{m \times b}$ have $\rank B = b$ and its column span be $\eta$-biased. Let $\map: \F_q^m \ra \F_a^n$ be a $\lambda$-pseudorandom puncturing. Then,
    \[
    \Pr[\map(B) \in \calM_{n, \tau}] \leq q^{n \left ( \log_q \E_{x \sim \Emp_B}[\tau(x)] + H_q(\tau) + \log_q \left ( 1 + \frac{\lambda q^b}{(1 - q^{2b} \eta)} \right ) + o_n(1) \right ) }.
    \]
\end{lemma}

\begin{proof}
We have that
\[
\Pr[\map(B) \in \calM_{n, \tau}] = \Pr[\Emp_{\map(B)} = \tau] \leq
q^{n \cdot \left ( -\DKLq(\tau \Vert \sigma)n + \log_q \left ( 1 + \frac{\lambda q^b}{(1 - q^{2b} \eta)} \right ) + o_n(1)\right )},\]
by \Cref{thm:biasedemp}, where $\sigma = \Emp_B$.
From here, as in \cite{GM21}, we attain the stated bound by using the concavity of log and the definition of $\DKLq$.
\end{proof}

We then recall two lemmas from \cite{GM21}, with no modification, which we use in the proof:
\begin{lemma}[Lemma 5.9~\cite{GM21}]\label{lemma:GM59}
    Let $B \in \F_q^{m \times b}$ have $\rank B = b$, and let $f: \F_q^b \ra \R$ be a non-negative function. Suppose that the column span of $B$ is $\eta$-biased, for some $\eta \geq 0$. Then, 
    \[
    \E_{x \sim \Emp_B}[f(x)] \leq \left ( 1 + q^b \eta \right ) \cdot \E_{x \sim U(\F_q^b)} [f(x)].
    \]
\end{lemma}

\begin{lemma}[Lemma 6.12~\cite{GM21}]\label{lemma:GM612}
    Let $n \in \N$, $q$ a prime power and $b \in \N$ such that $\frac{n}{\log_q n} \geq \omega_{n \ra \infty}(q^{2b})$. Let $\calC \subseteq \F_q^n$ be a linear code of rate $R \in [0, 1]$, sampled at random from some ensemble. Suppose that, for every $1 \leq a \leq b$, every distribution $\tau$ over $\F_q^a$ and every matrix $A \in \F_q^{Rn \times a}$ with $\rank A = a$, we have 
    \[
    \E_{\calC} \left [ \left | \{A \in \calM_{n, \tau} | A \subseteq \calC \} \right | \right ] \leq q^{(H_q(\tau) - a(1-R) + a \eps)n},
    \]
    for some fixed $\eps > 0$. Then, for every row-symmetric and $b$-local property $\calP$ over $\F_q^n$ such that $R \leq \text{RLC}(\calP) - 2\eps$, it holds that 
    \[
    \Pr_{\calC}[\calC \text{ satisfies } \calP] \leq q^{-n(\eps - o_{n \ra \infty}(1))}.
    \]
\end{lemma}

Lastly, we require one final lemma before we can conclude the final result:

\begin{lemma}\label{lemma:GM613}
    Fix $b \in \N$, and a full-rank distribution $\tau$ over $\F_q^b$. Let $\calD \subseteq \F_q^m$ be a $\eta$-biased linear code. Let $\map$ be a $\lambda$-pseudorandom $(m \ra n)$ puncturing map. Let $R = \frac{\log_q |\calD|}{n}$. Then,
    \[
    \E_{\calC}\left [ \left | \{A \in \calM_{n, \tau} | A \subseteq \calC \} \right | \right ] \leq q^{n \left (H_q(\tau) - (1-R)b + \log_q \left [ \left ( 1 + \eta q^b \right ) \left ( 1 + \frac{\lambda q^b}{(1 - q^{2b} \eta)} \right ) \right ] + o_n(1) \right )}.
    \]
\end{lemma}
\begin{proof}
    Let $\tau$ be a full-rank distribution over $\F_q^b$. From Lemma \ref{lemma:GM59}, we have that
    \[
    \E_{x \sim \Emp_B}[\tau(x)] \leq q^{-b} \left (1 + \eta q^b \right ),
    \]
    for all $B \in \F_q^{m \times b}$ such that $\rank B = b$ and $B \subseteq \calD$. From Lemma \ref{lemma:GM510},
    \[
    \Pr[\map(B) \in \calM_{n, \tau}] \leq q^{n \left ( \log_q \E_{x \sim \Emp_B}[\tau(x)] + H_q(\tau) + \log_q \left ( 1 + \frac{\lambda q^b}{(1 - q^{2b} \eta)} \right ) + o_n(1) \right ) }.
    \]

    Plugging in for $\E_{x \sim \Emp_B}[\tau(x)]$, we get that
    \begin{align*}
         \Pr[\map(B) \in \calM_{n, \tau}] &\leq q^{n \left ( \log_q \left ( q^{-b} \cdot \left ( 1 + \eta q^b \right ) \right ) + H_q(\tau) + \log_q \left ( 1 + \frac{\lambda q^b}{(1 - q^{2b} \eta)} \right ) + o_n(1) \right ) }\\
         &= q^{n \left (-b + H_q(\tau) + \log_q \left [ \left ( 1 + \eta q^b \right ) \left ( 1 + \frac{\lambda q^b}{(1 - q^{2b} \eta)} \right ) \right ] + o_n(1) \right ) }.
    \end{align*}
    Now, by taking a union bound over the at most $q^{Rnb}$ choices of $B$, we get that 
    \[
    \E_{\calC}\left [ \left | \{A \in \calM_{n, \tau} | A \subseteq \calC \} \right | \right ] \leq q^{n \left (H_q(\tau) - (1-R)b + \log_q \left [ \left ( 1 + \eta q^b \right ) \left ( 1 + \frac{\lambda q^b}{(1 - q^{2b} \eta)} \right ) \right ] + o_n(1) \right )}.
    \]
\end{proof}

Now, we can prove the main theorem of this section, by showing that for specific choices of $\lambda$, $\eta$, the bound from the previous lemma satisfies the conditions of Lemma \ref{lemma:GM612}.

\begin{proof}[Proof of \Cref{thm:GM1}]
Let $\eps$ be given. Let $\tau$ be a distribution over $\F_q^a$ with $a \leq b$. From \Cref{lemma:GM613}, we have that 
\begin{align*}
    \E_{\calC}\left [ \left | \{A \in \calM_{n, \tau} | A \subseteq \calC \} \right | \right ] &\leq q^{n \left (H_q(\tau) - (1-R)a + \log_q \left [ \left ( 1 + \eta q^a \right ) \left ( 1 + \frac{\lambda q^a}{(1 - q^{2a} \eta)} \right ) \right ] + o_n(1) \right )}\\
    &\leq q^{n \left (H_q(\tau) - (1-R)a + \frac{\eta q^a}{\ln q} + \frac{\lambda q^a}{\ln q \cdot (1 - q^{2a} \eta)} + \frac{\eta \lambda q^{2a}}{(1 - q^{2a} \eta ) \ln q} + o_n(1) \right )}.
\end{align*}

Now, note that by our substitutions $\eta = \frac{\eps \ln q}{4q^{2b + 1}}$ and $\lambda = \frac{\eps \ln q}{8q^b}$,
\begin{align*}
\frac{\eta q^a}{ \ln q} &= \frac{\eps b \ln q}{4q^{2b - a} \cdot q \cdot \ln q} \\
&\leq \frac{\eps b}{4q^{2b - a} \cdot q} \\
&\leq \eps a/4.
\end{align*}

Additionally, 
\begin{align*}
    \frac{\lambda q^a}{\ln q \cdot (1 - q^{2a} \eta)} &= \frac{\eps}{8q^{b-a} \cdot (1 - q^{2a} \cdot \frac{\eps \ln q}{4q^{2b+1}})} \\
    &= \frac{\eps}{8q^{b-a} \cdot (1 - \frac{\eps \ln q}{4q^{2b-2a+1}})} \\
    &\leq \frac{\eps}{8q^{b-a} \cdot (1 - \frac{\eps}{4q^{2b-2a}})}\\
    &\leq \frac{\eps}{8q^{b-a}(1 - \eps/4)} \leq \frac{\eps}{4}.
\end{align*}
Lastly, by combining the above two results, 
\[
\frac{\eta \lambda q^{2a}}{(1 - q^{2a} \eta ) \ln q} \leq \ln q \cdot \frac{\eps b}{4q^{2b - a} \cdot q} \cdot \frac{\eps}{4} \leq \frac{\eps^2 b}{16q^{2b-a}},
\]
where we have taken advantage of the fact that our expression is $\ln q$ multiplied by the two terms we have already bounded before. Thus, all three expressions are bounded by $a \cdot \frac{\eps}{4}$. As a result, 
\[
\E_{\calC}\left [ \left | \{A \in \calM_{n, \tau} | A \subseteq \calC \} \right | \right ] \leq q^{n \left (H_q(\tau) - (1-R)a + \frac{3}{4} \cdot a \eps + o_n(1) \right )}.
\]
We can then invoke Lemma \ref{lemma:GM612} to conclude our result for sufficiently large $n$.
\end{proof}

Now, by noting that list-decoding is a $O(1 / \eps)$-local property in our specific setting, we can conclude \Cref{thm:optLD} by using \Cref{thm:GM1} and \cite{GH11}.
That is:
\begin{proof}[Proof of \Cref{thm:optLD}]
    Suppose we fix $\rho \in (0, 1/2)$. Then, there exists a constant $\alpha > 0$ such that the threshold for an RLC being $(\rho, \alpha / \eps)$ list-decodable is $1 - H(\rho) - \eps$ \cite{GH11}.  We let $\calP$ denote the property of being $(\rho, \alpha / \eps)$-list decodable. This means that $b = O(1 / \eps)$. Now, let $\calD$ be a mother-code over $\F_q^m$ which is $\eta = \frac{\eps b \ln q}{4q^{2b+1}}$-biased (and note that the Hadamard code satisfies this property). From \Cref{thm:GM1}, we know that for a $\lambda = \frac{\eps \ln q}{8 q^b}$-pseudorandom puncturing of design rate $R \leq \text{RLC}(\calP) - \eps$ of $\calD$, for every $\eps > 0$:
    \[
    \Pr[\calC \text{ satisfies } \calP] \leq q^{(-\eps + o_{n \ra \infty}(1))n}.
    \]

    Thus, we can choose the design rate of the code to be $\text{RLC}(\calP) - \eps = 1 - H(\rho) - 2 \eps$, and with high probability, our code will still be $(\rho, \alpha / \eps)$ list-decodable. 

    By our choice in parameters for $\lambda$, we require the degree of the graph to be $O \left ( \frac{q^{2b}}{\eps^2 (\ln q)^2} \right )$. This means that every step will require $O(b \log q + \log 1 / \eps)$ random bits. Using the fact that $b = O(1 / \eps)$, this means that every step in the expander random walk of our pseudorandom puncturing requires $O( \left ( \log q \right ) / \eps )$ random bits. Because $q$ will be chosen to be a constant, this requires $O(1 / \eps)$ random bits per step. Initializing the random walk requires $\log m $ random bits, where $m$ is the length of the mother code $\calD$. Fortunately, $m \leq q^n$, so $\log m \leq n \log q$ (every generating matrix of length longer than $q^n$ will have duplicate rows). As such the total amount of randomness required is $O(n + n \cdot 1/ \eps) = O(n / \eps)$, as we desire.
\end{proof}

We can conclude the statement of Theorem \ref{thm:informalGeneral} almost identically, where we instead treat $b$ as a parameter, instead of substituting $O(1 / \eps)$.

\section{Random Noise Tolerance of PRLCs}\label{sec:random}
We show pseudorandom linear codes achieve capacity against the memoryless additive channel. Our proof follows directly from the argument of~\cite{GM21}, except we argue that a \emph{pseudorandom} puncturing approximately preserves the probability that a random vector lies in the code. In the context of this channel, we use the MLDU (maximum likelihood decoder under uniform prior). Upon receiving a corrupted codeword $z \in \F_q^n$, this decoder returns the codeword $x$ that maximizes
\[
\Pr[\text{receive } z | \text{original codeword was } x].
\]

\begin{theorem}\label{thm:MLDUDecode}
    Given a prime power $q$ and a distribution $X$ over $\F_q$ and $\eps \in (0,1)$, let $\calD\subseteq \F_q^m$ be an $\eps/8q$-biased linear code and let $\calC\subseteq \F_q^n$ be a $\eps/8q$-pseudorandom puncturing of $\calD$ with design rate $R\leq 1-H_q(X)-\eps$. Then there is a constant $c_X>0$ such that with probability $1-q^{c_X\cdot \eps n}$, for every $x\in \calC$ we have
    \[\Pr_{z\la X^n}[\text{MLDU decodes $x+z$ to $x$}] \geq 1-2q^{c_X\cdot \eps^2 n}.
    \]
\end{theorem}

Inspecting the proof of \cite[Theorem 6]{GM21} gives the following:
\begin{remark}{\label{rmk:MLDUDecode}}
    Let $q$ be a prime power, $\nu$ a distribution over $\F_q$, $\eps \in (0, 1)$, and $\calC \subseteq \F_q^n$ be a probabilistically constructed linear code of design rate $R \leq 1 - H_q(\nu) - \eps$ such that for every non-zero $x \in \F_q^n$
    \[
    \Pr_{\calC} \left [ x \in \calC \right ] \leq q^{n \cdot \left (-1 + R + \eps/4 \right )}.
    \]
    Then, with probability $1 - q^{- \Omega_{\nu}(\eps n)}$, it holds for all $x \in \calC$ that 
    \[
    \Pr_{z \sim \nu^n} \left [\text{the MLDU outputs } x \text{ on input } x + z \right ] \geq 1 - 2q^{-c_{\nu} \eps^2 n}.
    \]
\end{remark}

We will prove the following lemma, from which we can then immediately conclude \Cref{thm:MLDUDecode} by using \Cref{rmk:MLDUDecode}.
\begin{lemma}
    Let $\calC\subset \F_q^n$ be a $\lambda$-pseudorandom puncturing of a $\eta$-biased linear code $\calD\subset \F_q^m$. For every $x\in \F_q^n\setminus \{0\}$, we have
    \[\Pr[x\in \calC] \leq q^{n(R+\eta q+\lambda q-1)}.
    \]
\end{lemma}
\begin{proof}
    Fix an arbitrary nonzero $x$ and fix arbitrary $u\in \calD$. For $i \in [q]$ let 
    \[T_i := \{j \in [m]:u_j=i\}.\]
    and let $\tau_i := T_i/m$. By the bias of $\calD$, we have $\tau_i \leq \frac{1}{q}+\frac{q-1}{q}\eta$ for every $i$.
     
    Thus we have
    \begin{align*}
        \Pr[x\in \calC] &= \Pr[\map_1 \in T_{x_1}\wedge \ldots \wedge \map_n \in T_{x_n}]\\
        &\leq \left(\max_{i}\tau_i+\lambda\right)^n && \text{(\Cref{thm:EHS})}\\
        &\leq q^{n(\eta q + \lambda q-1)}
    \end{align*}
    and then a union bound over the $q^{Rn}$ codewords completes the proof.
\end{proof}

\begin{proof}[Proof of Theorem \ref{thm:MLDUDecode}]
    Let $\lambda = \eta = \frac{\eps}{8q}$. The bound from the previous lemma then states that 
    \[
    \Pr[x \in \calC] \leq q^{n(R + \eps/8 + \eps / 8 - 1)} = q^{n(R + \eps/4 -1)}.
    \]
    We can then invoke Remark \ref{rmk:MLDUDecode}.
\end{proof}

\section{Pseudorandom Puncturings of Large Distance Codes}\label{sec:RS}
We next show that pseudorandom puncturings of large distance codes are list recoverable beyond the Johnson Bound. We show this for the specific case of zero-error list-recovery, as it simplifies the exposition. 

\begin{theorem}\label{thm:LP12}
    Fix $\alpha \in (0,1]$. Let $\calD\subset \F_q^m$ be a linear code with distance at least $m(1-q^{-1}-\eps^2)$. Let $\map$ be a $1/4$-pseudorandom $(m\ra n)$ puncturing with $n=O(\log|\calD|/\eps)$. Then $\map(\calD)$ has rate $\Omega\left(\frac{\eps}{\log q}\right)$ and is $(\ell,\ell(1+\alpha))$-zero error list recoverable with high probability, assuming:
    \[1/\sqrt{q} \leq \eps \leq \min(c,\alpha/4), \quad \ell \leq \alpha/4\eps^2.\]
\end{theorem}

We state the main theorem that allows us to establish this, which is analogous to Theorem 3.1 of~\cite{LP20}.
\newcommand{\cD}{\mathcal{D}}
\begin{theorem}\label{thm:LP31}
    Given $\alpha \in (0,1)$, let $q,m,d,\ell,n \in \N$. Given a code $\cD\subset \F_q^m$ of distance at least $1-mq^{-1}-d$. Suppose that
    \[
    d \geq mq^{-1},\quad 4\alpha^{-1} \leq \ell \leq \alpha m/1600d, \quad n =\Omega\left(\sqrt{\ell/\alpha}\log|D|\right), \quad m>n.
    \]
    Then the probability that $C:=\map(D)$ (where $\map$ is a $1/4$-pseudorandom puncturing) is $(\ell,(1+\alpha) \ell)$-zero error list recoverable is at least $1-\exp(-\sigma n/100)$.
\end{theorem}
Our proof differs from that of~\cite{LP20} in two ways: our puncturing is pseudorandom, rather than truly random, and we argue about puncturings produced with replacement (which is natural in the setting of expander random walks which may revisit vertices). 

We first introduce some notation that will be used in the proof. 
\begin{definition}
    For an arbitrary code $\calC \subset \F_q^m$, let
    \[
    T(\calC) = \{i \in [m] \;|\; \exists c_1 \neq c_2 \in \calC, c_1[i] = c_2[i]  \}.
    \]
\end{definition}
We first argue that, given an index set $\map\in [m]^n$ such that $\map(\cD)$ is not list recoverable with the claimed parameters, there is a small subcode that fails to be.
\begin{lemma}\label{lem:LPlargeBadEvent}
    Let $\map\in [m]^n$ be such that $\map(\cD)$ fails to be $(\ell,\ell(1+\alpha))$-zero error list recoverable. Then there is a subcode $\calC' \subset \cD$ such that:
    \begin{itemize}
        \item $|\calC'| \leq 10\sqrt{\ell/\gamma}$
        \item $|\{i \in [n]:\map_i\in T(\calC')\}|\geq n/8$
    \end{itemize}
\end{lemma}
The proof of this lemma closely follows Theorem 3.1 in~\cite{LP20}, and as such we defer it to the appendix.
We furthermore require a concentration bound for the number of bad indices selected by the puncturing map, which is a simple consequence of the expander Chernoff bound.
\begin{lemma}\label{lem:LPexpChernoff}
    Let $B\subset [m]$ be a bad set of indices satisfying $\beta:=|B|/m\leq 1/16$, and let $\map$ be a $1/4$-pseudorandom puncturing.
    Then 
    \[\Pr \left [ \left | \{i \in [n]: \map_i \in B \} \right | \geq \frac{n}{8} \right ] \leq \exp(-\Omega(n)).\]
\end{lemma}
This follows from the expander Chernoff bound, as stated in \Cref{thm:expanderChernoff}. We can then prove \Cref{thm:LP31}.
\begin{proof}[Proof of \Cref{thm:LP31}]
Let $X$ be the indicator that $\map(\calD)$ fails to be $(\ell, \ell(1 + \alpha))$-zero error list recoverable. Then
\begin{align*}
    \E[X] &\leq \sum_{\calC' \subset \map(\cD):|\calC'|\leq 10\sqrt{\ell/\gamma}}\Pr[|\{i:\map_i\in T(\calC')\}|\geq \sigma n/4] && \text{(\Cref{lem:LPlargeBadEvent})}\\
    &\leq \sum_{\calC' \subset \map(\cD):|\calC'|\leq 10\sqrt{\ell/\gamma}} \exp(-\Omega(n)) && \text{(\Cref{lem:LPexpChernoff})}\\
    &\leq \exp\left(10\sqrt{\ell/\gamma}\log|\cD|-\Omega(n)\right)\\
    &\leq \exp(-\Omega(n))
\end{align*}
where the second line follows by observing 
\[
|T(\calC')| \leq d|\calC'|^2 \leq 100 d\ell/\gamma \leq \frac{m}{16}
\]
where the first inequality follows from the distance of the code and the third follows from our bound on $\ell$, and so $T(\calC')\subset [m]$ satisfies the properties of \Cref{lem:LPexpChernoff}.
\end{proof}

\begin{proof}[Proof of \Cref{thm:LP12}]
    Let $n:=\lceil c' \eps^{-1}\log|\calD|\rceil$ for come constant $c'>0$ such that the construct of \Cref{thm:LP31} is satisfied (for parameters to be chosen later). We first show that the actual rate of $\map(\calD)$ is equal to the design rate with high probability. Analogously to the proof of \Cref{lem:designRate}, note that this event is equivalent to there existing $u\in \calD$ such that $\map(u)=0$. Fixing arbitrary $u\in \calD$, let $T\subset [m]$ be the coordinates on which $u$ is zero. Then $|T|/m \leq 1/q+\eps^2 \leq .6$ by the absolute constraint on $\eps$ and that $q\geq 2$. Then
    \[\Pr[\map(u)=0] = \left(.6+\lambda\right)^n = 2^{-\Omega(n)}\]
    Thus the probability that all such codewords are not mapped to all zero indices is $|\calD|2^{-n}\leq \exp(-\Omega(n))$, so with high probability the rate of $\map(\calD)$ is equal to the design rate of $\Omega(\eps/\log q)$. Finally, choose $\ell = \alpha/4\eps^2$ and $d=\lfloor m\eps^2\rfloor$ and applying \Cref{thm:LP31} completes the proof. 
\end{proof}

\section{Acknowledgements}
We thank Madhu Sudan for helpful conversations.

\bibliographystyle{alpha}
\bibliography{ref}

\appendix
\section{Omitted Proofs}
\begin{proof}[Proof of \Cref{lem:LPlargeBadEvent}]
    By assumption, for $i\in [n]$ there are subsets $A_i\subseteq [q]$ such that $|A_i|\leq \ell$ and, letting
    \newcommand{\BAD}{\textsc{BAD}}
    \[
    \BAD :=\{c\in \cD:\map(c)\in \prod_{i\in [n]}A_i\},
    \]
    we have $|\BAD|\geq \ell(1+\alpha)$. 
    Let $\calC'\subset\BAD$ be defined by randomly including each element of $\BAD$ with probability 
    \[p=\sqrt{\frac{2\ell}{\gamma}}\cdot \frac{(1+\alpha)}{|\BAD|}.\]
    Note that
    \[\E[|\calC'|]=p|\BAD|\leq \sqrt{8\ell/\gamma}.\]
    Now note that for every $i$, there are at least $\alpha\ell/2$ pairs $\{c_1,c_2\}\in \BAD$ with $c_1[\map_i]=c_2[\map_i]$, which holds as $|A_i|\leq \ell$ and $|\BAD|\geq (1+\alpha)\ell$.
    Thus, for every $i \in [n]$, 
    \[\Pr[\map_i \notin T(\calC')] = (1-p^2)^{\alpha\ell/2} <1/2
    \]
    Thus, 
    \[
    \E[|\{i:\map_i \in T(\calC')\}|] \geq n/4.
    \]
    
    Finally, define the random variable
    \[
    Y = \sum_{i \in [n]}\mathbb{I}[\map_i \in T(\calC')] - \frac{n}{8}\frac{|\calC'|}{\E[|\calC'|]}.
    \]
    We have $\E[Y]\geq n/4$, and hence there exists $\calC'$ such that $Y$ achieves its expectation, which can only occur when \[|\{i:\map_i \in T(\calC')\}|\geq n/4\]
    and $|\calC'|\leq 10\sqrt{\ell/\gamma}$, so we conclude.
\end{proof}

\section{Relationship of Zero-Error List-Recoverability Bounds to Unbalanced Expanders}\label{app:expcons}

We will require the definition of an unbalanced expander. 

\begin{definition}[Unbalanced Expander]
    A $(k, d, \eps)$-regular unbalanced expander is a bipartite graph on vertex set $V = L \cup R$, $|L| \geq |R| $, where the degree of every vertex in $L$ is $d$, and for every $S \subseteq L$ such that $|S| = k$, we have that $|N(S)| \geq D |S| (1 - \eps)$.
\end{definition}

Further, we will require a procedure that turns a code into a graph. 

\begin{definition}[Bipartite Graph of a Code]
    For a code $\calC \subseteq [q]^n$, we denote by $G(\calC)$ the bipartite graph with vertex set $\calC \cup ([n] \times [q])$. For an arbitrary $c = (c_1, \dots c_n) \in \calC$, we associate it with the neighbors $\{(1, c_1), \ldots,(n, c_n)\}$.
\end{definition}

In \cite{LP20}, the authors show the following result, by relating zero-error list-recoverability to expansion of a graph $G$:

\begin{theorem}\cite{LP20}\label{thm:LP2.7}
Let $q, n$ be sufficiently large integers and $\alpha \in (0, 1), \eps > q^{-1/2}$ be real numbers. For every code $\calD \subseteq [q]^m$ with relative distance $1 - 1/q - \eps^2$, there is a subset $S \subseteq [m]$ such that $|S| = O(\eps m \log q)$ such that $G(\calD_S)$ is a $\left (\alpha \eps^{-2}, |S|, \alpha \right )$-unbalanced expander. 
    
\end{theorem}

\cite{LP20} instantiate this result for degree $d$ Reed-Solomon codes, with $m = q$, $\eps = (d / q)^{-1/2}$. Thus, $n = \widetilde{O}(\sqrt{q})$.
Note that in this setting, we are only guaranteed the existence among $\binom{n^2}{n} = 2^{O(n \log n)}$ possible choices for the punctured set. With our construction, we can again use degree $d$ Reed-Solomon Codes, and recover that there exists such an unbalanced-expander among only $2^{O(n)}$ possible choices for the punctured set. 

\end{document}